\documentclass[11pt, oneside]{amsart}   	% use "amsart" instead of "article" for AMSLaTeX format
\usepackage[bottom=47mm]{geometry}                		% See geometry.pdf to learn the layout options. There are lots.
\usepackage{graphicx}				% Use pdf, png, jpg, or eps? with pdflatex; use eps in DVI mode
\usepackage{amssymb}
\usepackage{xcolor}
\usepackage{pgfplots,tikz}
\usetikzlibrary{patterns, shapes.geometric, intersections, calc}
\definecolor{imayou}{RGB}{154, 154, 235}
\definecolor{usuai}{RGB}{9, 150, 126}
\definecolor{persred}{RGB}{154,63,63}
\definecolor{sand}{RGB}{201, 177, 60}
\pgfplotsset{compat=1.18}
\usepackage{hyperref}
\usepackage{amsmath}
\usepackage{amsthm}
\usepackage{mathrsfs}
\usepackage{mathtools}
\usepackage[title]{appendix}
\mathtoolsset{showonlyrefs=true}
\usepackage{xcolor}
\usepackage{bbm}
\usepackage{enumerate, comment} 
\usepackage{ulem}
\usepackage[foot]{amsaddr}

\title[Superellipse damping]{Optimal decay for waves damped by superellipses}

\author{B. Achammer$^{1}$}
\author{Perry Kleinhenz$^{1,*}$}
\address[1]{Illinois State University, Mathematics Department}
\address[*]{Corresponding author, email: pbklein@ilstu.edu}
\theoremstyle{definition}
\newtheorem{assumption}{Assumption}
\theoremstyle{theorem}
\newtheorem{theorem}{Theorem}[section]
\newtheorem{lemma}[theorem]{Lemma}

\newtheorem{proposition}[theorem]{Proposition}
\theoremstyle{definition}

\newtheorem{remark}[theorem]{Remark}
\newtheorem{example}[theorem]{Example}
\numberwithin{equation}{section}

\newcommand{\Rb}{\mathbb{R}}

\newcommand{\C}{\mathbb{C}}

\newcommand{\Nb}{\mathbb{N}}
\newcommand{\Tb}{\mathbb{T}}

\newcommand{\ra}{\rightarrow}

\newcommand{\e}{\varepsilon}
\renewcommand{\d}{\delta}

\newcommand{\nm}[1]{\left| \left| #1 \right| \right|}

\newcommand{\lp}[2]{ \nm{#1}_{L^{#2}}}

\newcommand{\hp}[2]{\nm{#1}_{H^{#2}}}

\newcommand{\ltwo}[1]{\lp{#1}{2}}

\newcommand{\vphi}{\varphi}

\newcommand{\p}{\partial}

\newcommand{\supp}{\text{supp }}

\newcommand{\Lc}{\mathcal{L}}

\newcommand{\ti}{\widetilde}

\newcommand{\Sb}{\mathbb{S}}

\renewcommand{\Re}{\text{Re }}
\renewcommand{\Im}{\text{Im }}

\newcommand{\ltworp}[1]{\|#1\|_{L^2(\Rb_+)}}

\begin{document}
	\begin{abstract}
		Energy decay rates for solutions of the damped wave equation on the torus are known to be influenced by the geometry of the damped set and the growth properties of the damping. In this paper we produce lower bounds on energy decay rates for a class of damping which are positive on a superellipse and grow polynomially like the distance to the boundary of the superellipse. The energy decay rates we obtain depend explicitly on the exponent used to define the superellipse and the polynomial power. We show these rates are sometimes optimal. The proof adapts quasimodes from $y$-invariant damping using a simplification of the usual normal form argument.\\
		\smallskip
		%\noindent \textbf{Keywords.} damped wave equation, energy decay, damping geometry
	\end{abstract}
	\maketitle
    \vspace{-1cm}
	\section{Introduction}
	Let $W$ be a bounded nonnegative function which is not identically zero, on the flat two torus $\Tb^2 = [-\pi,\pi)^2/\sim$, and let $u$ solve the damped wave equation 
	\begin{equation}\label{eq:DWE}
		\begin{cases}
			(\p_t^2 -\Delta + W \p_t )u = 0, & (z,t) \in \Tb^2 \times \Rb \\
			(u,\p_t u)|_{t=0 }= (u_0,u_1) \in H^2(\Tb^2) \times H^1(\Tb^2).
		\end{cases}
	\end{equation}
	In this paper we show how the geometry of $\{W>0\}$ and the growth rate of $W$ determine polynomial energy decay of the form 
	\begin{equation}
		E(u,t):=\frac{1}{2}\int_{\Tb^2} |\nabla u(z,t)|^2 + |\p_t u(z,t)|^2 dz \leq C t^{-2\alpha} \left( \hp{u_0}{2}^2 + \hp{u_1}{1}^2 \right). \label{eq:stable}
	\end{equation}
	For $r_0 \in (0,\pi), n \geq 1$, we will consider damping positive on the superellipse, see Figure \ref{f:superellipse}
    \begin{equation}
        E_n=\{(x,y)\in\Tb^2: |x|^n +|y|^n < r_0^n\}.
    \end{equation}
    \vspace{-.5cm}
    \begin{figure}[h]
        \centering
             \begin{tikzpicture} 
  \draw[thick] (-1.5,-1.5) rectangle (1.5,1.5); \draw[sand, pattern=north east lines, pattern color=sand, domain=0:pi/2] plot ({cos(\x r)^(2/10)},{sin(\x r)^(2/10)}) -- plot ({-sin(\x r)^(2/10)},{cos(\x r)^(2/10)}) -- plot ({-cos(\x r)^(2/10)},{-sin(\x r)^(2/10)}) -- plot ({sin(\x r)^(2/10)},{-cos(\x r)^(2/10)}) -- cycle; 
   \end{tikzpicture}
    \hspace{.4cm}
    \begin{tikzpicture} 
  \draw[thick] (-1.5,-1.5) rectangle (1.5,1.5); \draw[sand, pattern=north east lines, pattern color=sand, domain=0:pi/2] plot ({cos(\x r)^(2/4)},{sin(\x r)^(2/4)}) -- plot ({-sin(\x r)^(2/4)},{cos(\x r)^(2/4)}) -- plot ({-cos(\x r)^(2/4)},{-sin(\x r)^(2/4)}) -- plot ({sin(\x r)^(2/4)},{-cos(\x r)^(2/4)}) -- cycle; 
   \end{tikzpicture}
          \hspace{.4cm}
    \begin{tikzpicture} 
  \draw[thick] (-1.5,-1.5) rectangle (1.5,1.5); \draw[sand, pattern=north east lines, pattern color=sand, domain=0:pi/2] plot ({cos(\x r)^(1)},{sin(\x r)^(1)}) -- plot ({-sin(\x r)^(1)},{cos(\x r)^(1)}) -- plot ({-cos(\x r)^(1)},{-sin(\x r)^(1)}) -- plot ({sin(\x r)^(1)},{-cos(\x r)^(1)}) -- cycle; 
   \end{tikzpicture}
    \hspace{.3cm}
    \begin{tikzpicture} 
  \draw[thick] (-1.5,-1.5) rectangle (1.5,1.5); \draw[sand, pattern=north east lines, pattern color=sand, domain=0:pi/2] plot ({cos(\x r)^(2/1.5)},{sin(\x r)^(2/1.5)}) -- plot ({-sin(\x r)^(2/1.5)},{cos(\x r)^(2/1.5)}) -- plot ({-cos(\x r)^(2/1.5)},{-sin(\x r)^(2/1.5)}) -- plot ({sin(\x r)^(2/1.5)},{-cos(\x r)^(2/1.5)}) -- cycle; 
   \end{tikzpicture}
        \caption{The superellipse $E_n$ for $n=10, 4, 2,$ and $ 1.5$.}
        \label{f:superellipse}
    \end{figure}
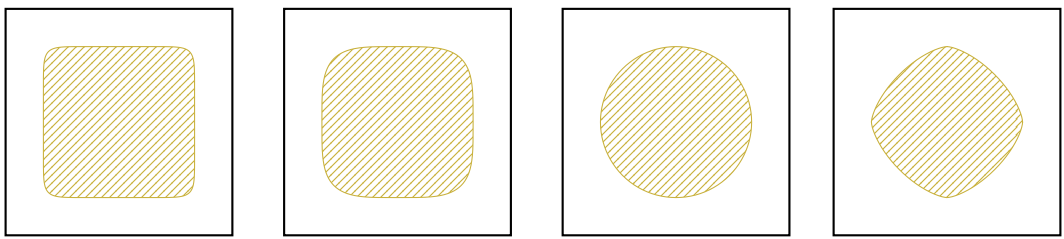

    Before stating our main result we state a preliminary assumption.
        \begin{assumption}\label{a:old}
        Assume $W$ is nonnegative, not identically zero, $W \in W^{9,\infty}(\Tb^2)$ and that there exists $C>0$ such that 
        \begin{equation}
            |\nabla W| \leq C |W|^{\frac{3}{4}},\quad  |\nabla^2 W| \leq C |W|^{\frac{1}{2}}.
        \end{equation}
    \end{assumption}
    \begin{theorem}\label{thm:first}
        For $\beta \geq 4, n \geq 2$, there exists $C>0$, $W$ nonnegative and bounded with 
        \begin{equation}\label{eq:distComp}
            \frac{1}{C} d((x,y), \Tb^2 \backslash E_n)^{\beta} \leq W(x,y) \leq C d((x,y), \Tb^2 \backslash E_n)^{\beta}, 
        \end{equation}
         such that \eqref{eq:stable} cannot hold with 
        \begin{equation}
            \alpha = 1- \frac{1}{\beta+\frac{1}{n}+3} +\d,
        \end{equation}
        for any $\d>0$. 
        
        Furthermore, if $\beta \geq 9$ and $n=2,4,6,8$ or $n \geq 9$, then this same $W$ also satisfies Assumption \ref{a:old}, and \eqref{eq:stable} holds with 
        \begin{equation}
            \alpha =1- \frac{1}{\beta+\frac{1}{n}+3}.
        \end{equation}
    \end{theorem}
    \begin{remark}
        \begin{enumerate}
            \item Note that the sharp energy decay rate improves as $n$ decreases. This is despite the fact that the damping region becomes smaller as can be seen in Figure \ref{f:superellipse}. Note also that as $n \ra \infty$, $E_n$ approaches a square and the energy decay rate approaches $1-\frac{1}{\beta+3}$ which matches that of damping on a strip \cite{DatchevKleinhenz2020} or a square \cite[Theorem 6.1]{Kleinhenz2025}. 
            \item Note that the sharp energy decay rate improves as $\beta$ increases. This result provides further evidence for the idea that the geometry of the damped set and the growth rate of the damping determine the sharp energy decay rate for solutions of the damped wave equation.
            \item This is a generalization of the second part of \cite[Theorem 1.1]{Sun23}, which provided the same result but only when $\beta \geq 10,$ and $n=2$.
            \item Our result complements \cite{DKP25} and \cite{Kleinhenz2025} which show, in this case, that for any damping satisfying Assumption \ref{a:old} and \eqref{eq:distComp} with $\beta \geq 9, n>0$, then \eqref{eq:stable} holds with 
            \begin{equation}
                \alpha = 1- \frac{1}{\frac{\beta}{\min(1,n)}+\frac{1}{n}+3}.
            \end{equation}
            Our result shows that this $\alpha$ cannot be improved for $\beta \geq 9$ and $n=2,4,6,8,$ or $n \geq 9$.
        \end{enumerate}
    \end{remark}

    \subsection{Literature Review}
	Polynomial decay rates have been studied for damping on the square in \cite{LiuRao2005} and on partially rectangular domains (including tori) in \cite{BurqHitrik2007}.
	
	In the setting of \eqref{eq:DWE}, whenever $\{W>0\}$ is nonempty, \eqref{eq:stable} holds with $\alpha = \frac{1}{2}$: see  \cite[Theorem 2.3]{AL14} and~\cite{AM14, BurqZworski2019}. 
	On the other hand, if some geodesics do not intersect $\supp W$, then \eqref{eq:stable} does not hold for any $\alpha > 1$  \cite[Theorem~2.5]{AL14}.
	
	Making additional assumptions on $W$ refines this rate. If $W(x,y)=(|x|-r_0)_+^{\beta}$, $\beta>-1$, near $\{x = r_0\}$, then \eqref{eq:stable} holds with $\alpha=1 - \frac{1}{\beta+3}$, and there are solutions decaying no faster than this rate \cite{Kleinhenz2019, DatchevKleinhenz2020, KleinhenzWang2026}. 
	That is, for $y$-invariant damping supported on a strip the polynomial growth of the damping near $\partial \{W>0\}$ determines the sharp polynomial energy decay rate of solutions. See \cite{Stahn2017} and \cite[Appendix B]{AL14} for earlier work where $W$ is exactly the indicator function of a strip.
	
	If $W$ satisfies Assumption \ref{a:old}, $\{W>0\}$ is a locally strictly convex set with positive curvature, and $W(x,y)=d((x,y),\{W=0\})^{\beta}$, then the sharp energy decay rate is faster. In the energy decay rate, $\beta$ is replaced by $\beta+\frac{1}{2}$, so \eqref{eq:stable} holds with $\alpha=1 - \frac{1}{\beta+\frac{1}{2}+3}$, and there are solutions decaying no faster than this rate \cite{Sun23}. 
    
    This improvement to the energy decay rate was generalized to damping growing polynomial-logarithmically and extended to $\{W>0\}$ equal to a super-ellipse, which has zero curvature at some points \cite[Theorem 6.3]{Kleinhenz2025}. These improvements to energy decay rates were further generalized in \cite{DKP25} to allow for more relaxed and general geometric assumptions on $\{W>0\}$. Roughly speaking, the positive curvature assumption is replaced by the assumption that near undamped closed geodesics $\p \{W>0\}$ can be locally written as the graph of $y=|x|^n$. For $n > 0$, $\beta$ is replaced by $\frac{\beta}{\min(n,1)}+\frac{1}{n}$ so \eqref{eq:stable} holds with $\alpha = 1 - \frac{1}{\frac{\beta}{\min(n,1)}+\frac{1}{n}+3}$. The main contribution of this paper is to provide a lower bound on energy decay rates which shows the decay rates from \cite{Kleinhenz2025, DKP25} are sometimes sharp.

	\subsection{Proof Outline}
    We first introduce some preliminary definitions and assumptions to discuss the outline of the proof.
	\begin{assumption}\label{a:new}
		Assume $W$ is nonnegative, bounded, not identically zero, and that there exists $C>0$ such that
		\begin{align}
            &|\p_x W| + |\p_x^2W| + |\p_y W| \leq C W^{\frac{1}{2}}.
		\end{align}
	\end{assumption}
    Let $(f)_+ =\max(0,f)$ and 
    \begin{equation}
        \psi(x) = \begin{cases} \frac{r_0-|x|}{r_0^n-|x|^n},  &|x|\neq r_0 \\
        \frac{1}{n r_0^{n-1}}, & |x|=r_0.
        \end{cases}
    \end{equation}
    For $\e\in(0,\min(\pi-r_0, \frac{r_0}{2}))$ we will assume that 
	\begin{align}\label{eq:Wdef}
		W(x,y) &=c_{\beta,n}(r_0^n-|x|^n -|y|^n)^{\beta}_+ \psi(x)^{\beta+\frac{1}{n}}, \quad r_0-\e < |x| \leq \pi,\\
		c_{\beta,n} &:=\pi\left(\int_0^1 (1-\rho^n)^{\beta} d \rho \right)^{-1}.
	\end{align}
     We define the average of $W$ in the $y$ direction 
    \begin{equation}
        A(W)(x):= \frac{1}{2\pi} \int_{-\pi}^{\pi} W(x,y) dy.
    \end{equation}
    We will also assume that
    \begin{equation} \label{eq:avgBelow}
         A(W)(x)\geq c >0, \quad |x| \leq r_0-\e.
    \end{equation}	
    See Figure \ref{f:3dDamp} for an illustration of a typical damping satisfying these assumptions.   
    \begin{figure}[h]
        \centering
        \includegraphics[width=0.5\linewidth]{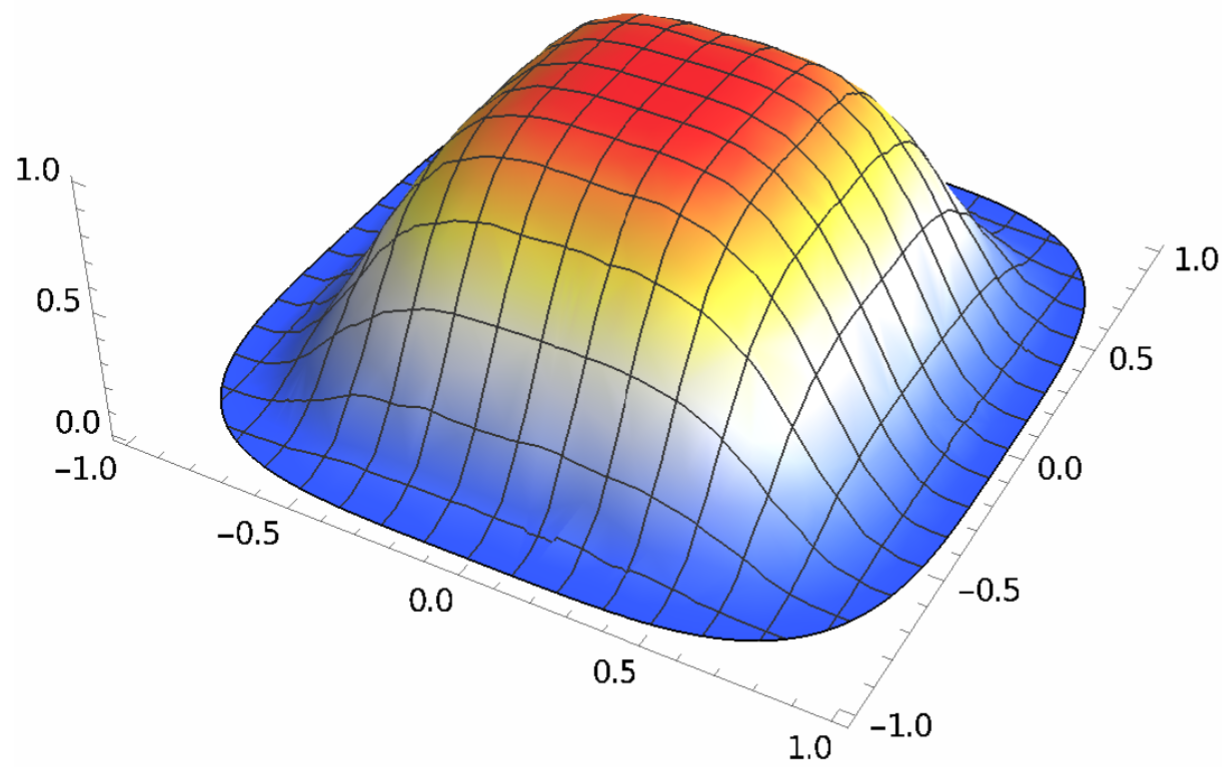}
        \caption{A typical damping satisfying Assumption \ref{a:new}, \eqref{eq:Wdef} and \eqref{eq:avgBelow}.}
        \label{f:3dDamp}
    \end{figure}    
    We now provide an explicit example of $W$ satisfying these assumptions and which we will use to prove Theorem \ref{thm:first}.
    \begin{example}\label{ex:explicit}
	    Let $\varphi \in C^{\infty}_c((-\pi,\pi):[0,1])$ satisfy $\varphi \equiv 1$ for $||x|-r_0|<\e$ and $\varphi \equiv 0$ for $|x|<r_0-2\e$, then define 
        \begin{equation}
            W(x,y)= c_{\beta,n}(r_0^n-|x|^n -|y|^n)^{\beta}_+ \left(1-\varphi(x) + \varphi(x) \psi(x)^{\beta+\frac{1}{n}}\right).
        \end{equation}
        
        %Note that the role of the cutoff $\varphi$ is to remove the singularity of $r_0-|x|$ from $\psi(x)$ at $x=0$. 
	\end{example}
    Note that this $W$ has $\{W>0\}=E_n$. When $\beta \geq 4, n \geq 2$ this $W$ satisfies Assumption \ref{a:new}, \eqref{eq:Wdef} and \eqref{eq:avgBelow}. Note also when $\beta \geq 9,$ and $ n=2,4,6,8,$ or $n \geq 9$, then this $W$ satisfies Assumption \ref{a:old}. Therefore to prove Theorem \ref{thm:first} it remains to show that such a $W$ satisfies \eqref{eq:distComp} and produces the appropriate energy decay. 
    \begin{proposition}\label{p:compDist}
    Suppose $W$ is as in Example \ref{ex:explicit}, then \eqref{eq:distComp} holds. That is, there exists $C>0$ such that 
    \begin{equation}
        \frac{1}{C}d((x,y), \Tb^2 \backslash E_n)^{\beta} \leq W(x,y) \leq C d((x,y),\Tb^2 \backslash E_n)^{\beta}.
    \end{equation}
    \end{proposition}
    We prove this in Section \ref{s:distance}.
    
    We now state our result providing a lower bound on the polynomial energy decay rate.
	\begin{theorem}\label{thm:main}
		Consider $W$ satisfying Assumption \ref{a:new}, \eqref{eq:Wdef} with $\beta \geq 4, n\geq 1$, and \eqref{eq:avgBelow},
		then \eqref{eq:stable} cannot hold with 
		\begin{equation}
			\alpha \geq 1- \frac{1}{\beta+\frac{1}{n}+3} +\d,
		\end{equation}
		for any $\d>0$.
	\end{theorem}
    \begin{remark}
        Note that we only require $n \geq 1$, in contrast to Example \ref{ex:explicit}. In Example \ref{ex:explicit}, note that $|x|^n$ does not have sufficient regularity near $x=0$ to satisfy Assumption \ref{a:new} for $n<2$.
    \end{remark}
    To prove this we apply \cite[Theorem 2.4]{BorichevTomilov2010}, as stated in \cite[Proposition 2.4]{AL14}, which states that energy decay as in \eqref{eq:stable} is equivalent to the following resolvent estimate: there exist $C,\lambda_0>0$ such that 
	\begin{equation}\label{e:resest}
		\nm{\left(-\Delta+ i\lambda W - \lambda^2 \right)^{-1}}_{\Lc(L^2(\Tb^2))} \leq C |\lambda|^{\frac{1-\alpha}\alpha},  \qquad \text{for }|\lambda| \geq \lambda_0, \lambda \in \Rb.
	\end{equation}
	Therefore to prove Theorem \ref{thm:main} it suffices to prove that there exists $C>0$ and $\lambda_j \ra \infty$ such that 
	\begin{equation} \label{eq:resolventLowerBound}
		\nm{\left(-\Delta+ i\lambda_j W - \lambda_j^2 \right)^{-1}}_{\Lc(L^2(\Tb^2))} \geq C |\lambda_j|^{\frac{1}{\beta+\frac{1}{n}+2}}.
	\end{equation}
	We do this by producing quasimodes in Section \ref{s:quasimode}.    

	To prove the energy decay rate in Theorem \ref{thm:first} we invoke \cite[Theorem 6.3]{Kleinhenz2025} which we recall a simplified version of here. See also  \cite[Theorem 1.10]{DKP25},
    \begin{theorem}[\cite{Kleinhenz2025}]
        Suppose $W$ satisfies Assumption \ref{a:old} and \eqref{eq:distComp} for some $\beta>0$ and $n \geq 2$, then \eqref{eq:stable} holds with 
    \begin{equation}
        \alpha = 1- \frac{1}{\beta+\frac{1}{n}+3}.
    \end{equation}
    \end{theorem}
    Since for $\beta \geq 9$ and $n=2,4,6,8,$ or $n \geq 9$, $W$ as defined in Example \ref{ex:explicit} satisfies Assumption \ref{a:old} and, by Proposition \ref{p:compDist}, \eqref{eq:distComp}, the energy decay rate in the second part of Theorem \ref{thm:first} is an immediate consequence of this theorem.

    \textbf{Acknowledgments:} The authors are grateful to Kiril Datchev and Antoine Prouff for helpful conversations. 
    
	\textbf{Funding:} This work was supported by the National Science Foundation [DMS-2530465 to P.K.].

	\section{Proof that $W$ is approximately a power of the distance to the boundary}\label{s:distance}
	% We begin with an immediate consequence of the mean value theorem, which we will repeatedly use.
	% \begin{lemma}\label{l:MVT}
	% 	Consider $n \geq 1$ and $0\leq a<b$. Then,
	% 	\begin{equation}
	% 		na^{n-1}\leq \frac{b^n-a^n}{b-a}\leq nb^{n-1}.
	% 	\end{equation}
	% \end{lemma}
	% \begin{proof}
	% 	The statement is immediate for $n=1$. Let $c\in(a,b)$, then for $n>1$ by the mean value theorem, we have 
	% 	\begin{align*}
	% 		na^{n-1} <nc^{n-1}&=\frac{b^n-a^n}{b-a}<n b^{n-1}.
	% 	\end{align*}
	% \end{proof}
	We first show that the distance from $(x,y)$ to $\Tb^2\backslash E_n$ is comparable to $(r_0^n-|x|^n-|y|^n)_+$.
	\begin{lemma}\label{l:comparable}
		Fix $n \geq 1$. There exists $C>0$, such that %for $(x,y) \in E_n$ %and in an $\e$ neighborhood of one of $\{(\pm r_0,0),(0,\pm r_0)\}$,  
		\begin{equation}
			C^{-1} d((x,y),\Tb^2 \backslash E_n) \leq (r_0^n-|x|^n-|y|^n)_+ \leq C d((x,y),\Tb^2 \backslash E_n).
		\end{equation}
	\end{lemma}
	\begin{proof}
    Note that when $d((x,y),\Tb^2 \backslash E_n)=0$ then $(r_0^n-|x|^n-|y|^n)_+=0$. Thus for $\eta>0$ small enough, it is enough to separately consider cases where $d((x,y),\Tb^2 \backslash E_n) \geq \eta$ or $0<d((x,y),\Tb^2 \backslash E_n)<\eta$.

    1) The set $K_{\eta} =\{(x,y) \in E_n:d((x,y),\Tb^2 \backslash E_n) \geq \eta\}$ is compact. Consider 
    \begin{equation}
        q(x,y)=r_0^n-|x|^n-|y|^n.
    \end{equation}
    Since $q$ is continuous and $q(x,y)=0$ only for $(x,y) \in \p E_n$, there exists $c >0$ such that for all $(x,y) \in K_{\eta}$ we have
    \begin{align}
        &c \leq q(x,y) \leq r_0^n, \\
        &\eta \leq d((x,y),\Tb^2 \backslash E_n) \leq r_0.
    \end{align}
    Therefore 
    \begin{equation}
        \frac{c}{r_0}\leq \frac{q(x,y)}{d((x,y),\Tb^2 \backslash E_n)} \leq  \frac{r_0^n}{\eta}.
    \end{equation}
    After rearranging, this provides the desired comparability.
    
    2) When $0<d((x,y),\Tb^2 \backslash E_n)<\eta$ for $\eta>0$ sufficiently small, by the symmetry of $E_n$ about the coordinate axes and the lines $y=\pm x$, it suffices to consider only $(x_1,y_1)$ with $x_1>\frac{r_0}{4}$ and $0 \leq y_1$.
    Let $(x_0,y_0) \in \p E_n$ be the point on $\p E_n$ closest to $(x_1,y_1)$, that is 
		\begin{equation}
			d((x_1,y_1),\Tb^2 \backslash E_n) = d((x_1,y_1),(x_0,y_0)).
		\end{equation}
		Also let $\ti{x}_1 =(r_0^n-y_1^n)^{\frac{1}{n}}$ so that $(\ti{x}_1,y_1) \in \p E_n$. Since $x_1>0$ and $(x_1,y_1) \in E_n$ we have $x_1 \leq \ti{x}_1 \leq r_0$. Furthermore
		\begin{equation}
			q(x_1,y_1) = r_0^n -x_1^n-y_1^n = \ti{x}_1^n - x_1^n.
		\end{equation}
		By the Mean Value Theorem applied to $f(z)=z^n$, and since $\frac{r_0}{4} \leq x_1 \leq \ti{x}_1\leq r_0$, we have
		\begin{equation}
			n\left(\frac{r_0}{4}\right)^{n-1}(\ti{x}_1-x_1) \leq \ti{x}_1^n - x_1^n =q(x_1,y_1)\leq n r_0^{n-1}(\ti{x}_1-x_1).
		\end{equation}
		  Now since $(\ti{x}_1,y_1) \in \p E_n$ we have
		\begin{equation} \label{eq:qint1}
			d((x_1,y_1), \Tb^2 \backslash E_n) \leq d((x_1,y_1),(\ti{x}_1,y_1)) = \ti{x}_1-x_1 \leq \frac{1}{n}\left(\frac{4}{r_0}\right)^{n-1}
            q(x_1,y_1).
		\end{equation}
		To show the opposite inequality, note that since $n \geq 1$ we have $\nabla q \in L^{\infty}$. Since $E_n$ is convex, $q$ is Lipschitz with Lipschitz constant $\lp{\nabla q}{\infty}$. Using that $(x_0,y_0) \in \p E_n$ so $q(x_0,y_0)=0$, and that $q$ is Lipschitz we have
		\begin{equation}\label{eq:qint2}
			q(x_1,y_1) = |q(x_1,y_1)-q(x_0,y_0)| \leq \lp{\nabla q}{\infty} d((x_1,y_1),(x_0,y_0))= \lp{\nabla q}{\infty} d((x_1,y_1),\Tb \backslash E_n).
		\end{equation}
        Together \eqref{eq:qint1} and \eqref{eq:qint2} provide the desired comparability.
	\end{proof}
    We now complete the proof of Proposition \ref{p:compDist}.
    \begin{proof}[Proof of Proposition \ref{p:compDist}]
        Since 
        \begin{equation}
            W(x,y) = c_{\beta,n} (r_0^n - |x|^n-|y|^n)_+^{\beta}(1-\phi(x) + \phi(x) \psi(x)^{\beta+\frac{1}{n}}),
        \end{equation}
        by Lemma \ref{l:comparable} it suffices to show there exists $C>1$ such that  
        \begin{equation}
            C^{-1} \leq 1-\varphi(x) + \varphi(x) \psi(x)^{\beta+\frac{1}{n}} \leq C.
        \end{equation}
        Then since $0 \leq \vphi(x) \leq 1$, it is enough to show that there exists $C >1$ such that 
        \begin{equation}
            \frac{1}{C}  \leq \psi(x) \leq C.       
        \end{equation}
        To show this we first compute for $x \in [0,r_0)$
    \begin{equation}
        \psi'(x) = \frac{d}{dx} \left(\frac{r_0-x}{r_0^n-x^n}\right)=\frac{nx^{n-1}(r_0-x) - (r_0^n-x^n)}{(r_0^n-x^n)^2}.
    \end{equation}
    For $n>1$ by the Mean Value Theorem, there exists $z \in (x,r_0)$ such that $r_0^n-x^n=(r_0-x) n z^{n-1}$. Thus 
    \begin{equation}
        n x^{n-1}(r_0-x) - (r_0^n-x^n) = n (x^{n-1} - z^{n-1})(r_0-x) <0.
    \end{equation}
    For $n=1$ we have $\psi'(x)=0$, thus $\psi'(x) \leq 0$ for $0\leq x<r_0$, and $\psi$ is decreasing there. Since $\psi$ is continuous on $[0,r_0]$ we have
    \begin{equation}
         \frac{1}{n r_0^{n-1}} =\psi(r_0)\leq \psi(x) \leq  \psi(0)=\frac{1}{r_0^{n-1}}.
    \end{equation}
    This establishes the desired bound on $\psi$ and so $W$ is comparable to a power of the distance to $\Tb^2 \backslash E_n$ as desired.
    \end{proof}

	\section{Quasimode construction}\label{s:quasimode}
	In this section we prove \eqref{eq:resolventLowerBound} by producing sequences $w_j \in H^2(\Tb^2), \lambda_j \in \Rb$ such that $|\lambda_j| \ra \infty$ and 
	\begin{equation}
		\ltwo{(-\Delta+i\lambda_j W-\lambda_j^2)w_j} \leq  C\lambda_j^{\frac{-1}{\beta+\frac{1}{n}+2}} \ltwo{w_j}.
	\end{equation}
	We take a semiclassical rescaling. Let $h_j=\lambda_j^{-1}$ and we will drop the $j$ subscript, writing $w_j=w_{h_j}=w_h$. So, it is enough to find $h \ra 0$, and  $w_h \in H^2(\Tb^2)$ such that 
	\begin{equation}
		(-h^2 \Delta + i h W - 1) w_h = O_{L^2}\left(h^2 h^{\frac{1}{\beta+\frac{1}{n}+2}}\right).
	\end{equation}
	
	The idea is to replace the damping $W$ by its average in the $y$ direction
	\begin{equation}
		A(W)(x)  = \frac{1}{2\pi}\int_{-\pi}^\pi W(x,y)dy.
	\end{equation}
	Then using \cite{Kleinhenz2019} we obtain a sequence of quasimodes $u_h$ for $(-h^2\Delta +ih A(W)-1)$ which we can adjust to quasimodes for the original problem. We do this with a simplification of the normal form argument of \cite[Section 7]{Sun23}. Specifically, we will let 
    \begin{align}
        &w_h(x,y) = e^{g(x,y)} u_h(x,y), \\
        &g(x,y):= \frac{1}{2} \int_{-\pi}^y (W(x,y')-A(W)(x)) dy'.
    \end{align}
    We are able to do this and avoid the use of microlocal analysis because the quasimodes of \cite{Kleinhenz2019} are already frequency localized in the $y$-direction.

	\subsection{Preliminary Lemmas}
	We begin by explicitly computing the average of our damping.
	\begin{lemma}\label{l:avgdamp}
		Suppose $W$ satisfies  \eqref{eq:Wdef} for $\beta > 0$, then for $|x| > r_0-\e$ we have
		\begin{equation}
			A(W)(x)=(r_0-|x|)^{\beta+\frac{1}{n}}_+.
		\end{equation}
	\end{lemma}
	\begin{proof}
    For $|x| \geq r_0$, it is immediate that $A(W)(x)=0$.
		By symmetry in $x$, it suffices to consider $r_0-\e < x < r_0$. Then by symmetry of $|y|^n$ and the definition of $W$ we have
		\begin{align}
			A(W)(x)=\frac{c_{\beta,n}}{\pi} \left(\frac{r_0-x}{r_0^n-x^n}\right)^{\beta +\frac{1}{n}} \int_0^{(r_0^n-x^n)^{\frac{1}{n}}} (r_0^n-x^n-y^n)^\beta dy.\label{eq:averageIntermed}
		\end{align}
		Letting $\kappa=(r_0^n-x^n)^{\frac{1}{n}}$ and $\rho=\frac{y}{\kappa}$, then we have
		\begin{align}
			\int_0^{(r_0^n-x^n)^{\frac{1}{n}}} (r_0^n-x^n-y^n)^\beta dy&=\int_{0}^{\kappa}(\kappa^n-y^n)^\beta dy
			=\kappa^{n\beta}\int_{0}^{\kappa}\left(1-\left(\frac{y}{\kappa}\right)^n\right)^\beta dy \ . \\
			&=\kappa^{n\beta+1}\int_{0}^{1}\left(1-\rho^n\right)^\beta d\rho = (r_0^n-x^n)^{\beta+\frac{1}{n}} \frac{\pi}{c_{\beta,n}}.
		\end{align}
		Plugging this back into \eqref{eq:averageIntermed} gives the desired conclusion.
	\end{proof}
	
	\subsection{$y$-invariant Quasimodes}
	We provide a summary of the quasimode construction in \cite{Kleinhenz2019}, where the damping depends only on $x$. We recall these details because we will take advantage of the separation of variables structure of these quasimodes, as well as their form on the damped set $\{|x| \leq r_0\}$. We roughly follow the exposition of \cite[Section 7.1]{Sun23}. The general strategy was first introduced in \cite[Appendix B]{AL14}.
	\begin{proposition}\label{p:oldquasi}
		Fix $\gamma>0$ and $0<\e<r_0$. Suppose $\ti{W}(x)=(r_0-|x|)^{\gamma}_+$ for $|x| > r_0-\e$ and $\ti{W}(x)\geq c>0$ for $|x| \leq r_0-\e$. Consider $F(x;\theta) \in H^1(\Rb_+)$ solving 
		\begin{equation} \label{eq:Fdef}
			\begin{cases}
				-F''+ix^{\gamma} F - \theta F =0, & x>0 \\
				F'(0)=1.
			\end{cases}
		\end{equation}
		Let $\chi(x) \in C^{\infty}([-\pi,\pi];[0,1])$ be even and satisfy $\chi(x)=0$ for $|x|<r_0-\e$ and $\chi(x)=1$ for $|x|>r_0-\frac{\e}{2}$.  For $k \in \Nb$, let 
        \begin{equation}
            h=\frac{\pi-r_0}{\sqrt{(\pi-r_0)^2k^2 +\frac{\pi^2}{4}}}, \qquad \rho_h=h^{\frac{1}{\gamma+2}}.
        \end{equation}
		Then for $k$ sufficiently large, there exist $\alpha_h = -\frac{\pi}{2(\pi-r_0)}+O(1)$, and $\mu_h=\frac{\pi h}{2(\pi-r_0)} + O(h \rho_h)$ such that the functions
		\begin{equation}
			u_h(x,y)= e^{iky} v_h(x) = e^{iky} \left[ \cos\left(\frac{\mu_h}{h}(\pi-|x|) \right) \mathbbm{1}_{r_0 < |x|<\pi} + \rho_h \alpha_h \chi(x) F\left( \frac{r_0-|x|}{\rho_h}; h^{\frac{-2(\gamma+1)}{\gamma+2}} \mu_h^2 \right) \mathbbm{1}_{|x| \leq r_0}  \right],
		\end{equation}
		satisfy $\ltwo{u_h} \approx 1$, $u_h \in H^2(\Tb^2)$, and are quasimodes satisfying 
		\begin{equation}
			(-h^2 \Delta + i h \ti{W} - 1)u_h = O_{L^2}(h^2 \rho_h). \label{eq:1dQuasimode}
		\end{equation}
	\end{proposition}
	\begin{proof}
		Plugging $u_h(x,y) =  e^{iky} v_h(x)$ into \eqref{eq:1dQuasimode}, it suffices to produce $v_h \in H^2(\Sb^1)$ satisfying 
		\begin{equation}
			(-h^2 \p_x^2 + i h \ti{W}- 1+h^2 k^2) v_h=  O_{L^2}(h^2 \rho_h), \quad \ltwo{v_h} \approx 1.
		\end{equation}
		Equivalently for $\mu_h=\frac{\pi h}{2(\pi-r_0)} +O(h\rho_h)$ it suffices to find quasimodes for
		\begin{equation}
			(-h^2 \p_x^2 + i h \ti{W} - \mu_h^2)v_h= O_{L^2}(h^2 \rho_h).
		\end{equation}
		To do so let $v_h(x)$ be even and defined separately where the damping is 0 or positive 
		\begin{equation}
			v_h(x) = \begin{cases}
				v_{h,l}(x), & 0 \leq x <r_0 \\
				v_{h,r}(x), &r_0 \leq x < \pi.
			\end{cases}
		\end{equation}
		Then $v_{h,r}=\cos(\frac{\mu_h}{h}(\pi-x))$ is the exact solution where $\ti{W}=0$. 
		To ensure the resulting function is $H^2$ we impose a compatibility condition at $x=r_0$
		\begin{equation}
			v_{h,l}(r_0)= \cos\left(\frac{\mu_h}{h}(\pi-r_0)\right), \qquad v_{h,l}'(r_0) = \frac{\mu_h}{h} \sin\left(\frac{\mu_h}{h}(\pi-r_0)\right).\label{eq:compat}
		\end{equation}    
		The left function must be a quasimode for
		\begin{equation}
			(-h^2 \p_x^2 + i h (r_0-|x|)_+^{\gamma}-\mu_{h}^2)v_{h,l} =O_{L^2}(h^2 \rho_h). 
		\end{equation}
		Now we set $v_{h,l}(x) = \rho_h \alpha_h \chi(x) F(\frac{r_0-x}{\rho_h}; h^{\frac{-2(\gamma+1)}{\gamma+2}} \mu_h^2)$, for $F$ solving \eqref{eq:Fdef} and where $O(1)=\alpha_h \in \C$ will be chosen to satisfy the compatibility conditions. Then computing directly and applying \cite[Lemma 3.1]{Kleinhenz2019} for $h$ sufficiently small, $v_{h,l}$ is an $O_{L^2}(h^2\rho_h)$ quasimode. Note that the cutoff $\chi$ is needed to ensure that $v_{h,l}$ satisfies the appropriate boundary condition at $x=0$ to permit an $H^2$ extension to $x<0$.
		
		It remains to check the compatibility conditions \eqref{eq:compat}.  Let $q_0$ be the smallest Neumann eigenvalue of $(-\p_x^2 +x^{\gamma})$ on $L^2(\Rb_+)$ and let $F_0=F(x;0)|_{x=0}$ be the Dirichlet trace at $\theta=0$. Then by \cite[Lemma 4.1]{Kleinhenz2019}, $q_0>0$. By Lax-Milgram, \cite[Lemma 4.2]{Kleinhenz2019} there exists a uniform $C_0>0$ such that for all $|\theta| \leq \frac{q_0}{2}$ the unique solution $F(x;\theta)$ of \eqref{eq:Fdef} satisfies 
		\begin{equation}
			\frac{1}{C_0} \leq |F(0;\theta)| \leq C_0.
		\end{equation}
		In particular $F_0=F(0;0) \neq 0$. Note that by taking $h$ small enough
        \begin{equation}\label{eq:thetasize}
             h^{\frac{-2(\gamma+1)}{\gamma+2}}\mu_h^2 = O(h^{\frac{2}{\gamma+2}}) \leq \frac{q_0}{4},
        \end{equation}
        so the $F$'s we work with satisfy these inequalities.        
        
        We make the ansatz for $\mu_h$ more specific,
		\begin{equation}
			\mu_h = \frac{\pi h}{2(\pi-r_0)} + h \rho_h \eta_h,
		\end{equation}
		where $\eta_h=O(1) \in \C$. Now using the definition of $v_{h,l}$ in terms of $F$ and applying a trigonometric identity, the compatibility conditions \eqref{eq:compat}become
		\begin{align}
			\alpha_h F(0;h^{\frac{-2(\gamma+1)}{\gamma+2}} \mu_h^2) &= -\rho_h^{-1} \sin((\pi-r_0)\rho_h\eta_h ) \label{eq:worsecompat1}\\
			-\alpha_h &= \left(\frac{\pi}{2(\pi-r_0)} + \rho_h \eta_h \right) \cos((\pi-r_0) \rho_h\eta_h ). \label{eq:worsecompat2}
		\end{align}
		Taking a Taylor expansion of $\sin$ and $\cos$ on the right hand side, the leading term of $\alpha_h$ should be $\alpha_0=-\frac{\pi}{2(\pi-r_0)}$ and the leading term of $\eta_h$ should be $\eta_0=\frac{\pi}{2(\pi-r_0)^2}F_0$. Using the implicit function theorem \cite[Lemmas 4.3, 4.4]{Kleinhenz2019}, for all $0 \leq h \ll 1$ there is a solution $(\alpha_h,\eta_h)$ to \eqref{eq:worsecompat1} and \eqref{eq:worsecompat2} with 
		\begin{equation}
			|(\alpha_h,\eta_h)-(\alpha_0,\eta_0)| = O(1).
		\end{equation}
        Finally, as we show in the following lemma, there is a uniform $C>0$ such that for $h$ small enough 
        \begin{equation}
            \ltwo{F(x; h^{\frac{-2(\gamma+1)}{\gamma+2}} \mu_h^2)} \leq C.
        \end{equation}
        Therefore by a direct computation $\ltwo{v_h} \approx \ltwo{u_h} \approx 1$.
        Thus the $u_h$ as defined in the statement of the Lemma are the desired $O_{L^2}(h^2\rho_h)$ quasimodes.
	\end{proof}
	We now prove a uniform bound on the $H^1$ norm of solutions of \eqref{eq:Fdef}. Our argument follows that of \cite[Lemma 7.2]{Sun23}.
	\begin{lemma}\label{l:Funiform}
	   Let $q_0$ be the smallest Neumann eigenvalue of $(-\p_x^2+x^{\gamma})$ on $L^2(\Rb_+)$. Then there exists $C>0$ such that for all $|\theta| \leq \frac{q_0}{4}$ the solution $F(x; \theta)$ of \eqref{eq:Fdef} satisfies $\|F\|_{H^1(\Rb_+)} \leq C$.
       
        Furthermore if $\gamma \geq 1$, then $\ltworp{x^{\frac{\gamma}{2}}\p_x F} \leq C$
	\end{lemma}
	\begin{proof}
		Take $J \in C^{\infty}_c(\Rb_+)$ such that $J(0)=0$ and $J'(0)=1$, then let $\ti{F}=F-J$. Therefore 
		\begin{equation}
			-\ti{F}'' + i x^{\gamma} \ti{F} - \theta \ti{F}=G,  \qquad G:=J'' - i x^{\gamma}J + \theta J.
		\end{equation}
		Pairing this equation with $\ti{F}$, integrating by parts, and then taking the real and imaginary parts we have 
		\begin{align}
			&\ltworp{\p_x \ti F}^2 \leq  |\Re \theta| \ltworp{\ti{F}} ^2+ \ltworp{G} \ltworp{\ti{F}} \\
			&\ltworp{x^{\frac{\gamma}{2}} \ti{F}}^2 \leq |\Im \theta| \ltworp{\ti{F}}^2 + \ltworp{G} \ltworp{\ti F}. \label{eq:Funiformintermed1}
		\end{align}
		Now since $-\p_x^2+x^{\gamma}-q_0 \geq 0$ and $|\theta| \leq \frac{q_0}{4}$, adding the two above inequalities we have 
		\begin{equation}
			q_0 \ltworp{\ti{F}}^2 \leq \ltworp{\p_x \ti{F}}^2 + \ltworp{x^{\frac{\gamma}{2}}\ti{F}}^2 \leq \frac{q_0}{2} \ltworp{\ti{F}}^2 + 2 \ltworp{G}\ltworp{\ti{F}}.
		\end{equation}
		Then applying Young's inequality, for some $C>0$ we have 
		\begin{equation}
			\ltworp{\ti{F}} + \ltworp{\p_x \ti{F}} \leq C \ltworp{G}\label{eq:Funiforminterm}.
		\end{equation}
		For $|\theta|\leq \frac{q_0}{4}$ there is uniform control over the size of  $\ltworp{G}$, which gives the desired $H^1$ bound.

        Now when $\gamma \geq 1$, let $Q=\p_x \ti{F}$ and note $Q$ solves
        \begin{equation}
            -Q''+ix^{\gamma} Q +i \gamma x^{\gamma-1}\ti{F} - \theta Q = \p_x G.
        \end{equation}
        Pairing this equation with $Q$, integrating by parts, taking the imaginary part, and computing directly we have 
        \begin{align}
            \ltworp{x^{\frac{\gamma}{2}}Q}^2 \leq &|\Im \theta| \ltworp{Q}^2 + \ltworp{\p_x G}\ltworp{Q}\\
            &+ \gamma \left(\ltworp{x^{\frac{\gamma}{2}}\ti{F}}\ltworp{x^{\frac{\gamma}{2}} Q} +\ltworp{\ti{F}} \ltworp{Q}\right).
        \end{align}
        Then using Young's inequality for products, recalling $Q=\p_x \ti{F}$, applying \eqref{eq:Funiformintermed1} and \eqref{eq:Funiforminterm}, and using that $|\theta| \leq \frac{q_0}{4}$, for some $C>0$ we have
        \begin{equation}
            \ltworp{x^{\frac{\gamma}{2}} \p_x F} \leq C\left( \ltworp{G} + \ltworp{\p_x G} \right).
        \end{equation}
        Since $|\theta| \leq \frac{q_0}{4}$ there is uniform control over the size of $\ltworp{\p_x G}$, which gives the desired bound.
	\end{proof}
	
	\subsection{New quasimodes}
	Now we show that we can modify $u_h$, the sequence of quasimodes with damping $A(W)(x)$, to produce a sequence of quasimodes of the same order with damping $W(x,y)$.
    Let $W$ satisfy \eqref{eq:Wdef} and \eqref{eq:avgBelow}. Then by Lemma \ref{l:avgdamp},
    \begin{align}
        &A(W)(x) = (r_0-|x|)_+^{\beta+\frac{1}{n}}, &\text{for } |x| > r_0-\e \\
        &A(W)(x) \geq c>0, &\text{for } |x| \leq r_0-\e.
    \end{align}
    Thus we may apply Proposition \ref{p:oldquasi} with $\ti{W}(x)=A(W)(x)$, and $\gamma=\beta+\frac{1}{n}$.
	\begin{proposition}\label{p:quasimode}
		Consider $W$ satisfying Assumption \ref{a:new}, \eqref{eq:Wdef} with $\beta \geq 4, n\geq 1$, and \eqref{eq:avgBelow}.
		Suppose $u_h(x,y)=e^{iky} v_h(x)$ is the sequence of quasimodes from Proposition \ref{p:oldquasi} for $\ti{W}(x)=A(W)(x)$. Recall that $h=\frac{\pi-r_0}{\sqrt{(\pi-r_0)^2k^2+\frac{\pi^2}{4}}}$ and $\rho_h=h^{\frac{1}{\beta+\frac{1}{n}+2}}$.
		Let 
		\begin{equation}
			g(x,y) = \frac{1}{2}\int_{-\pi}^y \left( W(x,y') - A(W)(x) \right) dy',
		\end{equation}
		and define 
        \begin{equation}
            w_h(x,y) = e^{g(x,y)} u_h(x,y)= e^{g(x,y)} e^{iky} v_h(x).
        \end{equation}
        Then $w_h \in H^2(\Tb^2)$ and $\ltwo{w_h} \approx 1$. Furthermore, we have 
		\begin{equation}
			\ltwo{(-h^2\Delta+ihW-1) w_h} = O(h^2 \rho_h).
		\end{equation}
	\end{proposition}
	\begin{proof}
		Since $\ltwo{u_h} \approx 1$, and $g$ is bounded above and below due to the boundedness of $W$ and $A(W)$, we have $\ltwo{w_h} \approx 1$. Note that because $W$ satisfies Assumption \ref{a:new}, $g \in W^{2,\infty}(\Tb^2)$. This along with $u_h \in H^2(\Tb^2)$ means that $w_h \in H^2(\Tb^2).$
		
		Now we compute directly, using that $u_h$ satisfy \eqref{eq:1dQuasimode}, and the boundedness of $e^g$ to see
		\begin{align}
			(-h^2\Delta + i h W-1) e^g u_h &= [-h^2\Delta, e^g]u_h  +e^g(-h^2\Delta+ihA(W)-1) u_h \\
			&\qquad+ ihe^g(W-A(W)) u_h \\
			&= [-h^2\Delta, e^g]u_h + ih e^g(W-A(W))u_h +O_{L^2}(h^2 \rho_h). \label{eq:comm1}
		\end{align}
		Then computing the commutator, writing $D_z=-i\p_z$, we have
		\begin{equation}
			[-h^2\Delta, e^g]= h^2 (D_x^2 e^g)+ h^2 (D_y^2 e^g) + 2h (D_x e^g) hD_x + 2h (D_y e^g) hD_y. \label{eq:comm2}
		\end{equation}
		Computing directly, using that $u_h(x,y)=e^{iky} v_h(x)$ and $1-h^2k^2=(1-hk)(1+hk) =O(h^2)$, we have
		\begin{equation}
			2h(D_y e^g)h D_y u_h +i h e^g (W-A(W)) u_h = (1-hk) i h e^g (W-A(W)) u_h = O_{L^2}(h^3).
		\end{equation}
		Combining this with \eqref{eq:comm1} and \eqref{eq:comm2}, and using that $\rho_h = h^{\frac{1}{\beta+\frac{1}{n}+2}} \geq h$, we have
		\begin{equation} \label{eq:whIntermediate}
			(-h^2\Delta + i h W-1) w_h = (h^2(D_x^2 e^g)+ h^2 (D_y^2 e^g) + 2h (D_x e^g) h D_x ) u_h + O_{L^2}(h^2 \rho_h). 
		\end{equation}
		We now claim that 
		\begin{equation}\label{eq:whIntermediate1}
			h^2 \ltwo{(D_x^2 e^g) u_h}+ h^2 \ltwo{(D_y^2 e^g) u_h}+ h \ltwo{(D_x e^g) hD_x u_h}  = O(h^2 \rho_h).
		\end{equation}
		If this claim is true, then applying it to \eqref{eq:whIntermediate} immediately shows that $w_h$ are the desired quasimodes, so it remains to prove \eqref{eq:whIntermediate1}.
		Before doing so we prove two preliminary lemmas.
		
		First we show how averaging interacts with our derivative conditions.
		\begin{lemma}\label{l:avgDBC}
			Suppose $f$ satisfies Assumption \ref{a:new}
			\begin{enumerate}
				\item Then $A(f)$ satisfies Assumption \ref{a:new}.
				\item There exists $C>0$ such that for $0 \leq j \leq 2$ we have 
				\begin{equation}
					\left|\p_x^j \int_{-\pi}^y f(x,y') - A(f)(x) dy' \right| \leq C(A(f)(x))^{\frac{1}{2}}.
				\end{equation}
			\end{enumerate}  
		\end{lemma}
		\begin{proof}
			1) Since $A(f)(x)$ is $y$-invariant we only consider $x$ derivatives. Applying Assumption \ref{a:new}, then since $g(z)=z^{\frac{1}{2}}$ is concave, applying Jensen's inequality we have for $j=1,2$
			\begin{align}
				|\p_x^{j} A(f)(x)| &= \left| \frac{1}{2\pi}\int_{-\pi}^{\pi} \p_x^{j} f(x,s) ds \right| \leq C\frac{1}{2\pi}\int_{-\pi}^{\pi} f^{\frac{1}{2}}(x,s) ds \\
				&\leq C\left( \frac{1}{2\pi} \int_{-\pi}^{\pi} f(x,s) ds \right)^{\frac{1}{2}}= C A(f)(x)^{\frac{1}{2}}.
			\end{align}
			2) When $j=0$ the bound is immediate from the triangle inequality. For $j=1,2$ taking derivatives, and applying Assumption \ref{a:new} and part 1 we have 
			\begin{align}
				\left| \int_{-\pi}^y \p_x^j f(x,y') - \p_x^j A(f)(x)  dy' \right|
				 \leq C A(f)^{\frac{1}{2}}(x) + 2\pi |\p_x^j A(f)(x)| \leq C (A(f)(x))^{\frac{1}{2}}.
			\end{align}
		\end{proof}
		Now we prove a preliminary lemma that allows us to control terms involving quasimodes.
		\begin{lemma}
			\begin{align}
				%&\ltwo{(\p_x A(W)) \p_x u_h} = O(\rho_h), \label{l:dxuhEstimate1}\\
				&\ltwo{\left( \p_x \int_{-\pi}^y (W(x,y') - A(W)(x)) dy' \right) \p_x u_h } = O(\rho_h) \label{eq:dxuhEstimate2},\\
				&\ltwo{A(W)^{\frac{1}{2}} u_h} + \ltwo{W^{\frac{1}{2}} u_h} = O(h^{\frac{1}{2}} \rho_h^{\frac{1}{2}}).\label{eq:a0westimate} 
			\end{align}
		\end{lemma}
		\begin{proof}
			Recall from the quasimode construction that there exists $F$ solving \eqref{eq:Fdef} such that for $x \in \{A(W)(x) > 0\}$ and $\gamma = \beta+\frac{1}{n}$, we have
			\begin{equation} \label{eq:quasiConstruct}
				u_h(x,y)=  \rho_h \alpha_h \chi(x) e^{iky} F\left(\frac{r_0-|x|}{\rho_h}; h^{\frac{-2(\gamma+1)}{\gamma+2}} \mu_h^2 \right).
			\end{equation}
			% Furthermore by \eqref{eq:thetasize} and Lemma \ref{l:Funiform} there exists $C>0$ such that for $h>0$ small enough
   %          \begin{equation}\label{eq:Fbound}
   %              \ltworp{F} +\ltworp{z^{(\beta+\frac{1}{n})\frac{1}{2}} F'(z)} \leq C.
   %          \end{equation}			
			1) Since $W$ satisfies Assumption \ref{a:new}, by Lemma \ref{l:avgDBC} part 2 and \eqref{eq:quasiConstruct}, then applying Lemma \ref{l:Funiform}, making a change of variables $z=\frac{r_0-|x|}{\rho_h}$, and applying Lemma \ref{l:Funiform} again we have 
			\begin{align}
				\ltwo{ \left(\p_x \int_{-\pi}^y (W(x,y') - A(W)(x)) dy'\right) \p_x u_h} &\leq C \ltwo{(A(W)(x))^{\frac{1}{2}} \p_x u_h} \\
				&\leq C\ltwo{ (r_0-|x|)_+^{(\beta+\frac{1}{n})\frac{1}{2}} F'\left( \frac{r_0-|x|}{\rho_h}\right)} + O(\rho_h)  \\
				&=O\left( \rho_h^{(\beta+\frac{1}{n})\frac{1}{2}} \rho_h^{\frac{1}{2}} \right)+O(\rho_h) \leq O(\rho_h),
			\end{align}
			where the final inequality follows from $\beta \geq 4 \geq 2$.
			
			2) Pairing \eqref{eq:1dQuasimode} with $u_h$, integrating by parts and taking the imaginary part we obtain 
            \begin{equation}\label{eq:buestimate}
			\ltwo{A(W)^{\frac{1}{2}} u_h} = O(h^{\frac{1}{2}} \rho_h^{\frac{1}{2}}).
		      \end{equation}
            Now computing directly, since $|u_h(x,y)|=|e^{iky} v_h(x)|$, then applying \eqref{eq:buestimate} we have
			\begin{align}
				\ltwo{W^{\frac{1}{2}} u_h}
				& =  \left( \int_{-\pi}^{\pi} \int_{-\pi}^{\pi} W(x,y) |v_h(x)|^2 dx dy \right)^{\frac{1}{2}} \leq C \ltwo{A(W)^{\frac{1}{2}} u_h} = O(h^{\frac{1}{2}} \rho_h^{\frac{1}{2}}).
			\end{align}
			% 2) Pairing \eqref{eq:whIntermediate} with $w_h=e^g u_h$ and integrating by parts we have 
			% \begin{align}
				%     \ltwo{h\nabla w_h}^2 + i h \ltwo{W^{\frac{1}{2}} w_h}^2 - \ltwo{w_h}^2 
				%     &=h^2 \<(D_x^2e^g) u_h, e^g u_h\> 
				%     +h^2\<(D_y^2e^g)u_h,e^gu_h\> \\
				%      &\qquad+ 2h\<(D_x e^g) hD_x u_h, e^g u_h\> 
				%     +O_{L^2}(h^2 \rho_h).
				% \end{align}
			% Note that the $(D_x^2 e^g)$ and $(D_y^2 e^g)$ terms are purely real, so taking the imaginary part of both sides and dividing by $h$ we obtain 
			% \begin{equation}
				%  	\ltwo{W^{\frac{1}{2}} w_h}^2 \leq 2  |\<(D_x e^g) h D_x u_h, e^g u_h\>| + O_L^2(h \rho_h).
				% \end{equation}
			% Now computing $D_x e^g$ directly, then applying Cauchy-Schwarz and applying \eqref{eq:dxuhEstimate2} and that $\ltwo{u_h}=1$ we have 
			% \begin{align}
				%      \ltwo{W^{\frac{1}{2}} w_h}^2 &\leq Ch \left|\< \left(\int_{-\pi}^y \p_x W(x,y') -A(W)'(x) dy' \right) \p_x u_h, u_h\>\right| + O(h \rho_h) \\
				%     &\leq Ch \ltwo{\left(\int_{-\pi}^y \p_x W(x,y') -A(W)'(x) dy' \right) \p_x u_h}  \ltwo{u_h} + O(h \rho_h) \label{eq:a0whIntermed1} \\
				%     &= O(h \rho_h).
				% \end{align}
			% Taking square roots gives us the desired inequality.
		\end{proof}
		We can now estimate the terms in \eqref{eq:whIntermediate1}.
		
		1) Computing directly 
		\begin{equation}
			D_x^2 e^g = - e^g\left( \frac{1}{4} \left(\p_x \int_{-\pi}^y \left( W(x,y') -  A(W)(x) \right) dy' \right)^2 + \frac{1}{2} \p_x^2\int_{-\pi}^y  (W(x,y') -  A(W)(x)) dy' \right).
		\end{equation}
		Applying Lemma \ref{l:avgDBC} part 2, and using that $e^g$ is bounded and that $A(W)$ satisfies Assumption \ref{a:new} we have 
		\begin{equation}
			h^2 \ltwo{(D_x^2 e^g) u_h} \leq C h^2 \ltwo{A(W)^{\frac{1}{2}}u_h}.\label{eq:hdxeg}
		\end{equation}
		
		2) Computing directly, and using that $e^g$ is bounded,  then applying the triangle inequality and using that $W$ satisfies Assumption \ref{a:new} we have 
		\begin{align}
			h^2 \ltwo{(D_y^2 e^g) u_h} &\leq C h^2 \left( \ltwo{(W-A(W))^2 u_h} +  \ltwo{(\p_y W) u_h} \right) \\
			&\leq C h^2 \left( \ltwo{W^{\frac{1}{2}} u_h} + \ltwo{A(W)^{\frac{1}{2}} u_h} \right).\label{eq:hdyeg}
		\end{align}
		Combining \eqref{eq:hdxeg} and \eqref{eq:hdyeg}, applying \eqref{eq:a0westimate} and using that $\rho_h =  h^{\frac{1}{\beta+\frac{1}{n} +2}} \geq h$ we have 
		\begin{equation}
			h^2 \ltwo{(D_x^2 e^g u_h)} + h^2 \ltwo{(D_y^2 e^g) u_h} = O(h^2 h^{\frac{1}{2}} \rho_h^{\frac{1}{2}}) = O(h^2 \rho_h).
		\end{equation}
		3) Computing directly, then applying  \eqref{eq:dxuhEstimate2} we have
		\begin{align}
			\ltwo{h (D_x e^g) h D_x u_h} &\leq C h^2 \ltwo{\left(\p_x \int_{-\pi}^y  (W(x,y') - A(W)(x) ) dy' \right)\p_x u_h } = O(h^2 \rho_h).
		\end{align}
		Thus all of the terms in \eqref{eq:whIntermediate1} are $O(h^2 \rho_h)$, and by \eqref{eq:whIntermediate} the $w_h$ are the desired quasimodes. 
	\end{proof}
	
	\begin{remark}
		Note that in \cite[Section 7]{Sun23}, a similar approach is used to produce quasimodes $w_h$ from quasimodes $v_h$ and $u_h$. However in that case $e^g$ is replaced by $e^{G_h} \widetilde{\psi}(h^{\frac{1}{2}} \rho_h^{\frac{1}{2}} D_x)$ where 
		\begin{equation}
			G_h = \text{Op}_h^w\left( \frac{\psi_1(\eta)}{2\eta} \int_{-\pi}^y (W(x,y') -A(W)(x)) dy' \right),
		\end{equation}
		where $(\xi, \eta)$ are the frequency variables dual to $(x,y)$, $\psi_1(\eta)$ is a cutoff supported near $\eta=1$, and $\tilde{\psi}(\xi)$ is a cutoff supported near $\xi=0$. These additional frequency cutoffs are included to frequency localize the quasimodes to the $y$-direction. This is necessary in the proof of the resolvent estimate from that paper \cite[Section 5]{Sun23}, because the contradiction quasimodes are not known to be frequency localized. 
		However, the quasimodes constructed in \cite{Kleinhenz2019} are frequency localized in the $y-$direction. 
	\end{remark}

	\bibliographystyle{alpha}
	\bibliography{mybib}

\begin{thebibliography}{DKP25}

\bibitem[AL14]{AL14}
N.~Anantharaman and M.~L{\'e}autaud.
\newblock Sharp polynomial decay rates for the damped wave equation on the
  torus.
\newblock {\em Anal. PDE}, 7(1):159--214, 2014.
\newblock With an appendix by S. Nonnenmacher.

\bibitem[AM14]{AM14}
N.~Anantharaman and F.~Maci{\`a}.
\newblock Semiclassical measures for the schr{\"o}dinger equation on the torus.
\newblock {\em Journal of the European mathematical society}, 16(6):1253--1288,
  2014.

\bibitem[BH07]{BurqHitrik2007}
N.~Burq and M.~Hitrik.
\newblock Energy decay for damped wave equations on partially rectangular
  domains.
\newblock {\em Mathematical Research Letters}, 14(1):35--47, 2007.

\bibitem[BT10]{BorichevTomilov2010}
A.~Borichev and Y.~Tomilov.
\newblock Optimal polynomial decay of functions and operator semigroups.
\newblock {\em Mathematische Annalen}, 347(2):455--478, 2010.

\bibitem[BZ19]{BurqZworski2019}
N.~Burq and M.~Zworski.
\newblock {Rough controls for Schr\"odinger operators on tori}.
\newblock {\em Annales Henri Lebesgue}, 2:331--347, 2019.

\bibitem[DK20]{DatchevKleinhenz2020}
K.~Datchev and P.~Kleinhenz.
\newblock Sharp polynomial decay rates for the damped wave equation with
  {H}{\"o}lder-like damping.
\newblock {\em Proceedings of the American Mathematical Society},
  148(8):3417--3425, 2020.

\bibitem[DKP25]{DKP25}
K.~Datchev, P.~Kleinhenz, and A.~Prouff.
\newblock Geometry of wave damping on the torus.
\newblock {\em arXiv preprint arXiv:2509.05239}, 2025.

\bibitem[Kle19]{Kleinhenz2019}
P.~Kleinhenz.
\newblock {Stabilization Rates for the Damped Wave Equation with
  H\"older-Regular Damping}.
\newblock {\em Commun. Math. Phys.}, 369(3):1187--1205, 2019.

\bibitem[Kle25]{Kleinhenz2025}
P.~Kleinhenz.
\newblock Sharp energy decay rates for the damped wave equation on the torus
  via non-polynomial derivative bound conditions.
\newblock {\em arXiv preprint arXiv:2502.09745}, 2025.

\bibitem[KW26]{KleinhenzWang2026}
Perry Kleinhenz and Ruoyu~PT Wang.
\newblock Sharp polynomial decay for polynomially singular damping on the
  torus.
\newblock {\em Annals of PDE}, 12(1):6, 2026.

\bibitem[LR05]{LiuRao2005}
Z.~Liu and B.~Rao.
\newblock Characterization of polynomial decay rate for the solution of linear
  evolution equation.
\newblock {\em Zeitschrift f{\"u}r angewandte Mathematik und Physik ZAMP},
  56(4):630--644, 2005.

\bibitem[Sta17]{Stahn2017}
R.~Stahn.
\newblock Optimal decay rate for the wave equation on a square with constant
  damping on a strip.
\newblock {\em Zeitschrift f{\"u}r angewandte Mathematik und Physik}, 68(2):36,
  2017.

\bibitem[Sun23]{Sun23}
C.~Sun.
\newblock Sharp decay rate for the damped wave equation with convex-shaped
  damping.
\newblock {\em International Mathematics Research Notices}, 2023(7):5905--5973,
  2023.

\end{thebibliography}

\end{document}